\documentclass[12pt]{article}

\usepackage[T1]{fontenc}
\usepackage[utf8]{inputenc}
\usepackage{authblk}
\usepackage{amsmath, amssymb,amsthm}

\usepackage{geometry}

\usepackage[all]{xypic}

\usepackage{caption}
\usepackage{graphicx}
\usepackage{indentfirst}

\usepackage{color}
\usepackage{comment}
\usepackage{xpatch}

\usepackage[colorlinks,
linkcolor=blue,
anchorcolor=blue,
citecolor=red
]{hyperref}

\newtheorem{theorem}{Theorem}

\newtheorem{claim}{Claim}

\newtheorem{lemma}{Lemma}

\newtheorem{remark}{Remark}
\newtheorem{prop}{Proposition}

\newtheorem*{fact}{Fact}

\newtheorem{definition}{Definition}

\usepackage{galois}
\numberwithin{equation}{section}

\begin{document}
	\captionsetup[figure]{labelfont={bf},name={},labelsep=period}
	\pagestyle{plain}
	
	\title{Maximal first Betti number rigidity for open manifolds of nonnegative Ricci curvature }
	\author{Zhu Ye\thanks{Supported partially by  National Natural Science Foundation of China [11821101] and  Beijing Natural Science Foundation [Z19003].\\	Email address: 2210501006@cnu.edu.cn.}}	
	\affil{Department of Mathematics, Capital Normal University}
	\date{}
	\maketitle
\begin{abstract}
	Let $M$ be an open Riemannian $n$-manifold  with nonnegative Ricci curvature. We prove that if the first Betti number of $M$ equals $n-1$, then $M$ is flat. 
\end{abstract}
	
\section{Introduction}

Let $M$  be a complete $n$-dimensional Riemannian manifold with nonnegative $\text{Ricci}$ curvature. 

If $M$ is compact, a classical result in Riemannian geometry says that the first Betti number $b_1(M) \leq n$ and ``=" holds if and only if  $M$ is a flat torus  (\cite{bochner1}, \cite{bochner2}). 

If $M$ is open (i.e. complete and not compact),  it has been known that $b_1(M)\leq n-1$ (\cite{split}, \cite{anderson}). In this paper, we will prove the following rigidity result:

\begin{theorem}\label{rigidty}
	If $M$ is an open  Riemannian $n$-manifold with $\text{Ric}_M\geq 0$, then $b_1(M)=n-1$ if and only if
	$M$ is  flat with a soul $T^{n-1}$, a torus of dimension $n-1$. 
\end{theorem}

\begin{remark}
	Note that $M$  in Theorem \ref{rigidty} has only two possible diffeomorphism types,  $\mathbb{R}\times T^{n-1}$ or $\mathbb{M}^2\times T^{n-2}$, where $\mathbb{M}^2$ is an open  M\"obius band.  In the former case, $M$ is isometric to $\mathbb{R}\times T^{n-1}$ (see (2) of Proposition \ref{int}).
\end{remark}	
Let's briefly explain a reason for $b_1(M)\leq n-1$.  Let  $\pi:(\tilde{M},\tilde{p})\rightarrow (M,p)$ be the Riemannian universal covering. Let $G$ be a finitely generated subgroup of $\pi_{1}(M,p)$.  Given a finite set of symmetric generators $S=\{g_1,\cdots,g_k\}$ of $G$, the word length $|g|$ of an element $g\in G$ is defined as
$	|g|=\min\{l\,|\, g=g_{i_1}g_{i_2}\cdots g_{i_l}\}$.	Put $U(r)=\{g\in G\,|\, |g|\leq r\}$ and let $\#(U(r))$ be the number of elements in $U(r)$.	
By the packing argument used in Milnor \cite{milnor} and the use of the Dirichlet fundamental domain, Anderson \cite{anderson} concluded that
\begin{equation}\label{eqa}
	\#(U(r))\cdot \text{Vol}(B_r(p))\leq \text{Vol}(B_{cr}(\tilde{p}))
\end{equation}             
for some constant $c>0$ independent of $r$. Since $\lim\limits_{r\to\infty}\frac{\text{Vol}(B_r(\tilde{p}))}{r^n}\leq \omega_n$ (	where $\omega_n$ is the volume of the unit ball in $\mathbb{R}^{n}$) by Bishop-Gromov volume comparison and $\liminf\limits_{r\to\infty} \frac{\text{Vol}(B_r(p))}{r}>0$
by Yau \cite{yau2}, one concludes that every finitely generated subgroup of $\pi_{1}(M,p)$ has polynomial  growth of order $\leq n-1$, thus $b_1(M)\leq n-1$ (see Proposition \ref{betti}). 

Note that  Cheeger and Gromoll also showed that $b_1(M)\leq n-1$  as an application of their splitting theorem ( Theorem 4 in \cite{split}). The approach to the main results of this paper is 
based on the proof of Theorem 4 in \cite{split}.

Let  $M$ be an open $n$-manifold with $\text{Ric}_M\geq 0$. By Theorem \ref{rigidty}, $b_1(M)=n-1$ implies that $M$ has linear volume growth and  $\tilde{M}$ has Euclidean volume growth. In this paper, we classified such manifolds: 

\begin{theorem}(part of Theorem \ref{main})\label{part}
	Let $M$ be an open $n$-manifold with $\text{Ric}_M\geq 0$ and let $\tilde{M}$ be its Riemannian universal cover. Then $M$ has linear volume growth and $\tilde{M}$ has Euclidean volume growth if and only if $M$ is flat with an $n-1$ dimensional soul.
\end{theorem}
We use  the following notion of  orbit growth in the proof of Theorem \ref{rigidty}  and Theorem \ref{part}:	
\begin{definition}\label{def1}
	Denote by $\#(A)$   the number of elements in a set $A$.		 Let $(X,d)$ be a metric space and let $\text{Isom}(X)$  be its  isometry group. Let $\Gamma$ be a subgroup of $\text{Isom}(X)$. For every $ x\in X$, put $	D^{\Gamma}(x,r)=\{g\in\Gamma :d(x,g(x))\leq r\}$.
	
	Given $p\in\mathbb{R}_+$, we say  $\Gamma$ has polynomial orbit growth related  to $x$ of order $\geq p \,\,(\leq p)$, if and only if 
	\begin{equation*}
		\liminf\limits_{r\to\infty} \frac{\#(D^{\Gamma}(x,r))}{r^p}>0\,\,\,(	\limsup\limits_{r\to\infty} \frac{\#(D^{\Gamma}(x,r))}{r^p}<\infty).
	\end{equation*}
	We say  $\Gamma$ has polynomial orbit growth related  to $x$
	of order $> p\,\,(<p)$, if and only if 
	\begin{equation*}
		\lim_{r\to\infty} \frac{\#(D^{\Gamma}(x,r))}{r^p}=\infty\,\,(=0).
	\end{equation*}
\end{definition}

One may verify without difficulty that the  polynomial orbit growth
properties defined above  do not depend on the choice of the base point $x$.

We will prove that either the conditions of Theorem \ref{rigidty} or Theorem \ref{part} imply that  the deck transformation group  $\Gamma$ of the Riemannian universal covering $\pi:\tilde{M}\rightarrow M$ has orbit growth of order $\geq n-1$. But we also have:
\begin{theorem}\label{flat}Let $M$ be an open $n$-manifold with $\text{Ric}_M\geq0$. Let $\pi:\tilde{M}\rightarrow M$ be the Riemannian universal covering with  deck transformation group $\Gamma$. If $M$ is not flat, then  $\Gamma$ has  polynomial orbit growth of order $< n-1$. 
\end{theorem}	
So the manifolds in Theorem \ref{rigidty} and Theorem \ref{part} can only be flat. Theorem \ref{flat} is obtained by improving the volume estimate in the proof of  Theorem 4 of  Cheeger-Gromoll  \cite{split}. The proof of Theorem \ref{flat} is given in Section 2.

Given Theorem \ref{flat}, we now prove Theorem \ref{rigidty}:

\begin{proof}[Proof of Theorem \ref{rigidty}]
	The condition $b_1(M)=n-1$  implies that $\pi_{1}(M)$ has  a finitely generated subgroup of polynomial growth  of order $\geq n-1$ (see Proposition \ref{betti}).  One then  concludes that the deck transformation group of the Riemannian universal covering of $M$ has
	polynomial orbit growth of order $\geq n-1$ (Proposition \ref{milnor}). By Theorem \ref{flat}, $M$ can only be flat. The soul theorem and the classical Bochner rigidity then imply that a soul of $M$ is $T^{n-1}$.
\end{proof}
To prove Theorem \ref{part}, we need the following  Theorem:
\begin{theorem}\label{dual}	Let  $N$ be a complete Riemannian $m$-manifold. Let  $\pi:(\bar{N},\bar{x}_0)\rightarrow (N,x_0)$ be a normal covering with deck transformation group $G$.   Then	for every $r>0$ we have:
	\begin{align}
		\label{my}	\#(D^G(\bar{x}_0,2r))\cdot \text{Vol}(B_r(x_0))&\geq \text{Vol}(B_r(\bar{x}_0)),\\
		\label{and}	\#(D^G(\bar{x}_0,r))\cdot \text{Vol}(B_r(x_0))&\leq \text{Vol}(B_{2r}(\bar{x}_0)).
	\end{align}	
\end{theorem}	

\begin{remark}
	Inequality(\ref{and}) is the orbit version of Anderson's inequality(\ref{eqa}). 
	
	The  author  observed that the inverse inequality(\ref{my}) also holds for orbits.  Note that the similar result does not hold for the word length, i.e. the inequality like
	\begin{equation}\label{wrong}
		\#(U(cr))\cdot\text{Vol}(B_{cr}(x_0))\geq \text{Vol}(B_{r}(\bar{x}_0))
	\end{equation}   
	is false in general. This is the main reason why we use the notion of polynomial orbit growth instead of polynomial growth of a (finitely generated) group. 	For example, we may consider  $N=\mathbb{R}\times \mathbb{S}^1$ with warped product metric 
	$g_N=dr^2+\phi(r)ds^2 $ and its universal cover, where $ds^2$ is the canonical metric on $\mathbb{S}^1$, $\phi(r)=1$ for $|r|\leq 1$ and $\phi(r)=r^{-2}$ for $|r|>2$.	 A simple estimate shows that inequality(\ref{wrong}) does not hold for every $c>0$.

\end{remark}	

Let $M$  be an open  $n$-manifold with $\text{Ric}_M\geq0$. Let $\Gamma$ be the deck transformation group of its Riemannian universal covering. It is clear from the inequality(\ref{and}) that $\Gamma$ has polynomial orbit growth of order $\leq n-1$.

\begin{definition}
	We say $\Gamma$ has maximal orbit growth, if and only if $\Gamma$ has polynomial orbit growth of order $\geq n-1$, i.e. 
	\begin{equation*}
		\liminf\limits_{r\to\infty} \frac{\#(D^{\Gamma}(\tilde{x},r))}{r^{n-1}}>0 \text{ for some $\tilde{x}\in \tilde{M}$}.
	\end{equation*}	
\end{definition}

By Theorem \ref{flat}, if $\Gamma$ has maximal orbit growth, then $M$ is flat. In the following Theorem \ref{main}, we give several equivalent characterizations of the maximal orbit growth condition and classify such manifolds.

It is well known from the work of Milnor \cite{milnor} and Gromov \cite{expand} that if $\pi_1(M)$ is finitely generated, then it is almost nilpotent. The nilpotency rank of a finitely generated almost nilpotent group $G$ is   defined in 2.4.1 of \cite{nz} and is denoted by $\text{rank}(G)$, which equals the polycyclic rank of a finite index subgroup of $G$. Similar to the first Betti number, we have $\text{rank}(\pi_{1}(M))\leq n-1$ if $\pi_{1}(M)$ is finitely generated (see Proposition \ref{algebra2}).
\begin{theorem}\label{main}
	Let $M$ be an open $n$-manifold with $\text{Ric}_M\geq0$.  Let $\pi:\tilde{M}\rightarrow M$ be the Riemannian universal covering with  deck transformation group $\Gamma$.  Then we have:
	
	\noindent\text{(a) $M$ is flat with an $n-1$ dimensional soul.}
	
	\noindent if and only if  any one of the following conditions holds:
	
	\noindent(b) There is a finitely generated subgroup G of $\pi_{1}(M)$ such that $\text{rank}(G)= n-1$.
	
	\noindent(c) $\Gamma$ fails to have  polynomial orbit growth of order $< n-1$. That is, there exists a sequence  $r_i\to\infty$ such that 
	\begin{equation*}
		\lim_{i\to\infty} \frac{\#(D^{\Gamma}(\tilde{x},r_i))}{r_i^{n-1}}>0 \text{ for some $\tilde{x}\in \tilde{M}$} .
	\end{equation*}
	
	\noindent(d) $\Gamma$ has maximal orbit growth.
	
	\noindent(e) $M$ has linear volume growth and $\tilde{M}$ has Euclidean volume growth.
\end{theorem}

\begin{proof}
	$(a)\Rightarrow (b)$: The condition shows that $M$ has the fundamental group of a compact flat $(n-1)$-manifold. It follows from the Bieberbach theorem that $\mathbb{Z}^{n-1}$ is a finite index  subgroup of $\pi_{1}(M)$.			
	
	$(b)\Rightarrow (c)$: By Proposition \ref{algebra2}, $G$ has polynomial  growth of order $\geq n-1$. By Proposition \ref{milnor}, $G$ has polynomial orbit   growth of order $\geq n-1$. This implies $(d)$ holds, hence $(c)$ holds.
	
	$(c)\Rightarrow (d)$: If $(c)$ holds, Theorem \ref{flat} claims that $M$ is flat. This implies that $\Gamma$ has polynomial orbit growth of order $\geq k$ and $\leq k$, where $k$ is the dimension of a soul of $M$ (see (1) of Proposition \ref{int}).  Condition  (c) forces $k=n-1$. Hence (d) follows.
	
	$(d)\Leftrightarrow (e)$: If (d) holds, inequality(\ref{and}) of Theorem \ref{dual}, together with the fact that $M$ has at least linear volume growth show that $\tilde{M}$ has Euclidean volume growth.  Using inequality(\ref{and}) again, we find that $M$ has linear volume growth.
	
	If (e) holds, inequality(\ref{my}) of Theorem \ref{dual} shows that $\Gamma$ has maximal orbit growth.
	
	$(d)\Rightarrow (a)$: Theorem \ref{flat} shows that $M$ is flat. Since $M$ has maximal orbit growth, (1) of Proposition \ref{int}  asserts that  every soul of $M$  has dimension $n-1$.
\end{proof}

\section{Proof of Theorem \ref{flat}}

We prove Theorem \ref{flat} in this section. Let $N$ be a  complete Riemannian manifold.     For $p\in N,r>0$,
set \begin{align*}
	D_{r}(p)=&\{x\in N\,|\, d(x,p)\leq r \},\\
	C_{r}(p)=&\{x\in N\,|\,  \text{if $q\in N$ and $d(q,x)>r$, then $d(p,x)+d(x,q)-d(p,q)>0$}\}, \\
	N_r(p)=&D_r(p)\cup C_{r}(p).
\end{align*}	
Note that $C_r(p)$ is exactly the points $x\in N$ such that every minimal geodesic connecting $p$ and $x$ cannot extend  at the $x$ end to a minimal geodesic of length $>d(p,x)+r$. To prove Theorem 4 in \cite{split},	Cheeger-Gromoll made the following key observation:	

\begin{lemma}\label{1}
	Let $x_0\in N, R>0, \Lambda=\text{Isom}(N)$. If $N$ contains no line, then there exists  a $d>0$ such that $\Lambda\cdot B_R(x_0)\subset N_d(x_0)$. 
	
\end{lemma}
\begin{proof}
	Otherwise, there exist $ d_i\to\infty ,x_i\in B_R(x_0)$ and $f_i\in \Lambda$ such that	$f_i(x_i)\notin N_{d_i}(x_0)$, then  we can find  unit speed minimal geodesics $\gamma_i:[-d_i,d_i]\rightarrow N$ with $\gamma_i(0)=f_i(x_i)$. But $f_i^{-1}\circ\gamma_i$ converges to a line $\gamma$ with $d(x_0,\gamma)\leq R$. This is  a contradiction.
	
\end{proof}

The key observation of the author is contained in the following Lemma \ref{2}:
\begin{lemma}\label{2}Fix $h>0$. Let $N$ be an $m$-dimensional complete Riemannian manifold with nonnegative $\text{Ricci}$ curvature ($m\geq 2$). 	 For $p\in N,r>0$, set $W_r^h(p)=D_r(p)\cap C_{h}(p)$.  Then 
	\begin{equation*}
		\lim\limits_{r\to\infty}	\frac{\text{Vol}(W_r^h(p))}{r^{m-1}}=0.
	\end{equation*}
\end{lemma}
\begin{proof}
	Set
	\begin{align*}
		S_{p}N=&\{v\in T_{p}N\,|\,||v||=1 \},\\
		C_{p}N=&\{v\in S_{p}N\,|\, \text{exp}_{p}(tv)|_{[0,\infty)} \,\text{is not a ray}\},\\
		C^s_{p}N=&\{v\in C_{p}N\,|\,\text{exp}_{p}(tv)|_{[0,s+\epsilon)} \text{ is not a minimal geodesic for every $\epsilon>0$}\},\\
		g(s)=&m(C^s_{p}N),\text{where $m$ is the standard measure on $S_{p}N$.}
	\end{align*}
	
	Using the Bishop-Gromov volume element comparison in the polar coordinate of $T_pN$, we get 
	\begin{equation*}
		\begin{split}
			\text{Vol}(W_r^h(p))=&\int_{C^{r+h}_pN}\int_{a(\theta)}^{b(\theta)}\mu(t,\theta)dtd\theta \\
			\leq &\int_{C^{r+h}_pN}\int_{a(\theta)}^{b(\theta)}t^{m-1}dtd\theta\\
			\leq & m(S_pN)\cdot h\cdot r^{m-1},
		\end{split}
	\end{equation*}		
	where $\mu(t,\theta)dtd\theta$ is the volume form of $N$ at $\text{exp}_p(t\theta)$, $b(\theta)-a(\theta)\leq h$, and $0\leq a(\theta)\leq b(\theta)\leq r $ for every $\theta\in C_p^{r+h}N$ (This is exactly  the estimate  Cheeger-Gromoll obtained).
	
	The author observed that  by measure theory, $\lim\limits_{r\to\infty}m(C^{r+h}_{p}N\backslash C^{\sqrt{r}}_{p}N)=0$ since 
	\begin{equation*}
		m(C_{p}N)<\infty, C^{r_1}_{p}N\subset C^{r_2}_{p}N \,\,\text{for}\,\, r_1<r_2 ,\text{and} \,\,C_{p}N=\cup_{r>0}C^r_{p}N.
	\end{equation*}	
	Put $A_r=C^{r+h}_{p}N\backslash C^{\sqrt{r}}_{p}N$. Similar to (1), we have for $r>\max\{1,h\}$ 
	\begin{equation*}
		\begin{split}
			\text{Vol}(W_r^h(p)\backslash W^h_{\sqrt{r}}(p))=&\text{Vol}(C_h(p)\cap(D_r(p)\backslash D_{\sqrt{r}}(p)))\\
			=& \int_{A_r}\int_{a_1(\theta)}^{b_1(\theta)}\mu(t,\theta)dtd\theta \\
			\leq &\int_{A_r}\int_{r-h}^{r}t^{m-1}dtd\theta\\
			\leq & m(A_r)\cdot h\cdot r^{m-1},
		\end{split}
	\end{equation*}			
	where $b_1(\theta)-a_1(\theta)\leq h$ and  $\sqrt{r}\leq a_1(\theta)\leq b_1(\theta)\leq r $.	
	
	So 
	\begin{align*}
		\text{Vol}(W^h_r(p))&=\text{Vol}(W^h_{\sqrt{r}}(p))+\text{Vol}(W_r^h(p)\backslash W^h_{\sqrt{r}}(p))\\
		&\leq h(m(S_pN)r^{\frac{m-1}{2}}+m(A_r)r^{m-1})
	\end{align*}	
	The result follows.
	
\end{proof}

\begin{proof}[Proof of Theorem \ref{flat}]
	Fix a point $\tilde{p}\in \tilde{M}$. Since $M$ is not flat, we have by splitting theorem that	$(\tilde{M},\tilde{p})=(N^k\times\mathbb{R}^{n-k},(x_0,0)) $, where $N$ does not contain a line and  $k\geq 2$.  
	
	Fix an $l>0$ such that $B_l(g_1\cdot\tilde{p})\cap B_l(g_2\cdot\tilde{p})=\emptyset\,$ for any $ g_1,g_2\in\Gamma,g_1\neq g_2$. by Lemma \ref{1}, there is a $d>0$ such that $\Lambda\cdot B_l(x_0)\subset N_d(x_0)$, where $\Lambda=\text{Isom}(N)$. We conclude that
	\begin{equation}\label{eq1}
		\Gamma\cdot B_l(\tilde{p})\subset N_d(x_0)\times \mathbb{R}^{n-k} ,
	\end{equation}
	since the isometry group of $\tilde{M}$ also splits. Write
	\begin{equation*}
		\Gamma_r=D^\Gamma(\tilde{p},r)=\{g\in\Gamma\,|\,d(g\cdot \tilde{p},\tilde{p})\leq r \},
	\end{equation*}
	(\ref{eq1})  implies
	\begin{equation}\label{eq2}
		\sqcup_{g\in \Gamma_r}g\cdot B_l(\tilde{p})\subset (N_d(x_0)\cap B_{r+l}(x_0))\times B_{r+l}(0).
	\end{equation}
	By 	Lemma \ref{2}, we may write $\text{Vol}(W_r^d(x_0))=f(r)r^{k-1}$, where $\lim\limits_{r\to\infty}f(r)=0$. (\ref{eq2}) then gives 
	\begin{equation*}
		\begin{split}
			&\#(\Gamma_r)\cdot\text{Vol}(B_l(\tilde{p}))\\
			\leq& \omega_{n-k}\left( \text{Vol}(D_d(x_0))+\text{Vol}(W^d_{r+l}(x_0))\right)(r+l)^{n-k}\\ 	
			\leq& \omega_{n-k}(\omega_{k}d^k+f(r+l)r^{k-1} )\cdot (r+l)^{n-k},
		\end{split}
	\end{equation*}
	Since $k\geq 2$, this gives the result.
	
\end{proof}
\section{Realization of Maximal Orbit Growth}	
In this section, we prove the Propositions used in the Introduction and prove Theorem \ref{dual}. They are used to realize the maximal orbit growth under various  conditions.

To obtain the  polynomial growth property  of a finitely generated subgroup of $\pi_{1}(M)$, Milnor \cite{milnor} essentially used the orbit growth as a bridge:

\begin{prop}(Milnor \cite{milnor})\label{milnor}
	Let $X$ be a metric space,  $x_0\in X$, $\Gamma$  be a finitely generated subgroup of $\text{Isom}(X)$, and $S=\{g_1,\cdots,g_k\}$ be a set of symmetric generators of $\Gamma$. Put
	\begin{align*}
		W(r)&=\{g\in \Gamma\,|\, \text{the word length of $g$ related to } S, |g|\leq r  \},\\
		h&=\max\{ d(x_0,g_ix_0)\,|\,i=1,\cdots,k\}.
	\end{align*}
	Then $W(r) \subset  D^\Gamma(x_0,hr) $. Especially, if $\Gamma$ has polynomial growth of order $\geq $ p ($>$ p), then $\Gamma$ has 
	polynomial  orbit growth of order $\geq $ p ($>$ p).
\end{prop}  
\begin{proof}
	If $g\in W(r)$, we can write $g=g_{i_1}g_{i_2}\cdots g_{i_t}$, where $t\leq r$, $g_{i_j}\in S$. We have 
	\begin{align*}
		d(x_0, gx_0)&= d(x_0,g_{i_1}g_{i_2}\cdots g_{i_t}x_0)\\
		&\leq d(x_0,g_{i_1}x_0)+d(g_{i_1}x_0,g_{i_1}g_{i_2}x_0)+\cdots+d(g_{i_1}g_{i_2}\cdots g_{i_{t-1}}x_0,g_{i_1}g_{i_2}\cdots g_{i_t}x_0)\\
		&\leq hr.
	\end{align*}
	The result follows.		
\end{proof}	
Next, we show how the first Betti number of a space controls the polynomial growth of its fundamental group. Recall: 
\begin{definition}(\cite{grouptheory} section 7.2)
	If $G$ is an  abelian group, the $\text{rank}$ of $G$ is  the maximal integer $k$ such that there exist $g_1,g_2,\cdots,g_k\in G$ which satisfy : 
	\begin{equation*}
		\text{if} \,\,\sum_{i=1}^{k}l_ig_i=0\,\,\text{for some $l_1,\cdots,l_k\in\mathbb{Z}$}, \text{then}\,\, l_1=l_2=\cdots=l_k=0.
	\end{equation*}
	$g_1,g_2,\cdots,g_k$ is called independent in this case. Denote by $\text{rank}(G)$  the $\text{rank}$ of $G$.
\end{definition}
\begin{remark}
	If $G$ is a finitely generated abelian group, then its rank equals its nilpotency rank.
\end{remark}  
\begin{definition}
	Let $X$ be a topological space. The first Betti number of $X$, denoted by $b_1(X)$,
	is the  rank of its first homology group $H_1(X)$. 
\end{definition}

From the proof of Theorem 1.3 in \cite{anderson}, we have the following algebraic proposition:

\begin{prop}\label{betti}
	Let $X$ be  a path-connected space, $x_0\in X$. If $b_1(X)=k$, then there is a subgroup $G$ of $\pi_1(X,x_0)$ such that $G$ is generated by $k$ elements and $G$ has polynomial growth of order $\geq k$. 
\end{prop}

We include a proof  of Proposition \ref{betti}  here for the convenience of readers: 
\begin{proof}[Proof of Proposition \ref{betti}]
	Let   $h:\pi_1(X,x_0)\rightarrow H_1(X)$
	be  the Hurewicz homomorphism. Since $b_1(X)=k$, we can choose $\gamma_1,\cdots,\gamma_{k}\in H_1(X)$ such that they are independent. We can choose $g_i\in (h)^{-1}(\gamma_i)$ since  $h$ is surjective.  Set
	\begin{align*}
		S_1&=\{g_1,\cdots,g_{k},g_1^{-1},\cdots,g_{k}^{-1} \},\\
		S_2&=\{\gamma_1,\cdots,\gamma_{k},\gamma_1^{-1},\cdots,\gamma_{k}^{-1} \},\\
		G_i&=\langle S_i\rangle \,\,\text{for} \,\,i=1,2,\\
		W_i(r)&=\{g\in G_i\,|\, \text{the word length of $g$ related to } S_i, |g|\leq r  \}.
	\end{align*}
	For every $r>0$, there is an injection
	\begin{equation*}
		\begin{split}
			A_r:W_2(r)&\rightarrow W_1(r)\\
			l_1\gamma_1+\cdots+l_{k}\gamma_{k}&\mapsto g_1^{l_1}g_2^{l_2}\cdots g_{k}^{l_{k}}
		\end{split}
	\end{equation*}
	since $h\circ A_r=\text{id}$.	 This  implies that $G_1$ has polynomial growth of order $\geq k$ since $G_2$ has polynomial growth of order $\geq k$. 
	
\end{proof}

Similarly, we have:
\begin{prop}\label{algebra2}
	If $G$ is a finitely generated almost nilpotent group with $\text{rank}(G)=k$, then $G$ has polynomial growth of order $\geq k$. 
\end{prop}
\begin{proof}
	It is easy to verify that a finitely generated group $A$ has polynomial growth of order $\geq k$ if a finitely generated subgroup  of $A$ has polynomial growth of order $\geq k$.	By Theorem 17.2.2 of \cite{grouptheory}, there exists  a finite index torsion-free subgroup of $G$. Denote it by $H$. Then $H$ is finitely generated and $\text{rank}(H)=k$.  It suffices to prove that $H$ has polynomial growth of order $\geq k$. By Theorem 17.2.2 of \cite{grouptheory}, there is a normal series 
	\begin{equation*}
		H=H_k\rhd H_{k-1}\rhd\cdots\rhd H_0=\{e\}
	\end{equation*}	
	such that each $H_i/H_{i-1}$ is isomorphic to $\mathbb{Z}$. Denote  by $[h_i]=h_iH_{i-1}$ a generator of  $H_i/H_{i-1}$. It is easy to check that
	\begin{align*}
		I:\mathbb{Z}^k&\rightarrow H\\
		(l_1,\cdots,l_k)&\mapsto h_1^{l_1}\cdots h_k^{l_k}
	\end{align*}
	is an injection. If we choose a finite set of generators of $H$ containing $h_1,h_2,\cdots,h_k$, then the fact that $I$ is an injection implies that	$H$ has polynomial growth of order $\geq k$ with respect to this set of generators.

\end{proof}

We now study the orbit growth of open flat manifolds:
\begin{prop}\label{int} Let $M$ be an  open flat  $n$-manifold. Let 
	$\pi:(\mathbb{R}^n,0^n)\rightarrow (M,x_0)$ be the universal covering map, with deck transformation group $\Gamma$. then 
	
	\noindent(1) $\Gamma$ has polynomial orbit growth of order $\geq k$ and $\leq k$, where $k$  is the dimension of a soul $S$ of $M$. 
	
	\noindent(2) If $T^{n-1}$ is a soul of $M$, then $M$ is diffeomorphic to $\mathbb{R}\times T^{n-1}$ or $\mathbb{M}^2\times T^{n-2}$. In the former case, $M$ is isometric to $\mathbb{R}\times T^{n-1}$.
\end{prop}
\begin{proof}
	(1)	Denote the Sharafutdinov retraction (\cite{shara}, see also \cite{yim}) by $r:M\rightarrow S$. $r$ is a distance nonincreasing strong deformation retraction. Choose a base point $x_0\in S$ and let $\iota:S\hookrightarrow M$ be the inclusion. We  make the identification $\pi_1(M,x_0)=\pi_1(S,x_0)$ by $\iota_*:\pi_1(S,x_0)\rightarrow \pi_1(M,x_0)$.
	
	The properties of $S$ and  $r$ guarantee the following:
	\begin{fact}
		(a) 	$\pi^{-1}(S)$ with the induced metric is totally geodesic in $\mathbb{R}^n$ and $$\pi|_{\pi^{-1}(S)}:(\pi^{-1}(S), 0^n)\rightarrow (S,x_0)$$ is the Riemannian universal covering map of $S$. Hence $\pi^{-1}(S)$ is a $k$-dimensional linear subspace  of $\mathbb{R}^n$.
		
		(b) 	For every $\gamma\in \pi_1(S,x_0),y\in \pi^{-1}(S), \gamma\cdot y=\iota_*(\gamma)\cdot y$, where both deck transformations are determined by the basepoint $0^n$.
	\end{fact}				
	Now Bieberbach theorem tells us that the pure translation subgroup  $\mathbb{Z}^k$ has finite index in $\pi_1(S,x_0)$.  By Fact (b),  $\iota_*(\mathbb{Z}^k)\leq \Gamma$  has polynomial  orbit growth of order $\leq k$ and $\geq k$.  It follows that $\Gamma$ has polynomial orbit growth of order $\geq k$. By  Lemma \ref{control} below, we obtain that  $\Gamma$  also has polynomial orbit growth of order $\leq k$. 
	
	\begin{lemma}\label{control}Let $(X,d)$ be a metric space and let $H$ be a subgroup of $\text{Isom}(X)$. Let $K$ be a finite index subgroup of $H$. If $K$ has polynomial orbit growth of order $\leq k$, then $H$ also has polynomial orbit growth of order $\leq k$.
		
	\end{lemma}
	
	\begin{proof}
		We may write $H$ as the disjoint union of left cosets: $H=\sqcup_{i=1}^{l}h_iK$.
		Fix an $x_0\in X$. Set $r_0=\max\limits_{i=1,\cdots,l}\{d(x_0,h_ix_0)\}$. If $h=h_jb\in D^H(x_0,r) $ for some $b\in K$, then	
		\begin{equation*}
			d(bx_0,x_0)\leq d(bx_0,h_j^{-1}x_0)+d(h_j^{-1}x_0,x_0)\leq r+r_0
		\end{equation*}	
		implies that $b\in D^K(x_0,r+r_0)$. So $\#(D^H(x_0,r))\leq l\cdot\#(D^K(x_0,r+r_0))$. The result follows.	
	\end{proof}			
	(2) 
	In this case, $\pi_{1}(M,x_0)=\pi_{1}(S,x_0)\cong\mathbb{Z}^{n-1}$. Since every $\gamma\in \mathbb{Z}^{n-1}$ acts on $\mathbb{R}^n$ by isometry, we may write $\gamma=(A_\gamma,v_\gamma)\in \text{O}(n)\ltimes \mathbb{R}^n$. Put $\Lambda=\pi^{-1}(S)$. Since $\mathbb{Z}^{n-1}$ acts on  $\Lambda$ by translation, Fact (a) and (b) shows that $A_\gamma|_\Lambda=\text{id}$ for every $\gamma\in \mathbb{Z}^{n-1}$. Let $\Lambda^\perp$ be the orthonormal completment of $\Lambda$, then $\Lambda^\perp$ is a 1-dimensional invariant  subspace of every $A_\gamma$. Each $A_\gamma|_{\Lambda^\perp}$ can be the reflection or  the identity map. Let $e_1=(1,\cdots,0),\cdots,e_{n-1}=(0,\cdots,1)$ be the canonical generators of $\mathbb{Z}^{n-1}$. Write $e_i=(A_i,v_i)$. If every $A_i$ is the identity map on $\Lambda^\perp$, then $M$ is isometric to $\mathbb{R}\times \mathbb{T}^{n-1}$.  Otherwise, without loss of generality, we may assume  that $A_1,\cdots, A_l$ are the reflection on $\Lambda^\perp$, $A_{l+1},\cdots,A_{n-1}$ are the identity map on $\Lambda^\perp$, then 
	\begin{align*}
		\Gamma&=\langle e_1,\cdots, e_{n-1}\rangle \\
		&=\langle e_1,e_2-e_1,\cdots,e_l-e_1,e_{l+1},\cdots,e_{n-1}\rangle\\
		&=\langle (A_1,v_1),(\text{id},v_2-v_1),\cdots, (\text{id},v_l-v_1),(\text{id},v_{l+1}),\cdots,(\text{id},v_{n-1})\rangle	.
	\end{align*}	
	In this case, $M$ is diffeomorphic to $\mathbb{M}^2\times T^{n-2}$.	
	
\end{proof}

Finally, we prove Theorem \ref{dual}:

\begin{proof}[Proof of Theorem \ref{dual}] 
	Proof of inequality(\ref{my}): We use the notion of Dirichlet domain as in \cite{anderson}. For every $g\in G$, put $D_g=\{x\in\overline{N}\,|\,d(x,\bar{x}_0)<d(x,g\bar{x}_0)\}$. Define the Dirichlet domain $F$ associated to $\bar{x}_0$ by $F=\bigcap\limits_{g\in G,g\neq e} D_g$.  Then $F$ is a fundamental domain for the action of $G$. Denote by $\overline{F}$ the closure of $F$ in $\bar{N}$.
	
	By the proof of Theorem 1.1 in \cite{anderson}, the following hold:
	\begin{align*}
		\pi(B_r(\bar{x}_0)\cap \overline{F})&=B_r(x_0),\\
		\partial F=\overline{F}\backslash F \,\,&\text{has Riemannian measure 0 in $\overline{N}$},\\
		\text{Vol}(B_r(\bar{x}_0)\cap F)&=\text{Vol}(B_r(x_0)).
	\end{align*}

	\begin{claim}\label{claim1}
		$\bigcup\limits_{g\in D(\bar{x}_0,2r)}g\cdot (B_r(\bar{x}_0)\cap \overline{F})\supset B_r(\bar{x}_0)$.
	\end{claim}
	\begin{proof}
		For every $y\in B_r(\bar{x}_0)$, let $g_y\bar{x}_0$ be a point in $G\cdot\bar{x}_0$ such that $d(g_y\bar{x}_0,y)=\min_{g\in G}\{d(g\bar{x}_0,y)\}$. Then $d(g_y\bar{x}_0,y)\leq d(\bar{x}_0,y)<r$. So $g_y\in D^G(\bar{x}_0,2r)$ and $g_y^{-1}y\in B_r(\bar{x}_0)$. If $\bar{x}_0$ is the unique point in $G\cdot\bar{x}_0$ which is closest to $g_y^{-1}y$, then $g_y^{-1}y\in D_g$ for every $g\neq e$. By definition, $g_y^{-1}y\in F$, so $ y\in g_y\cdot(B_r(\bar{x}_0)\cap F)$. Otherwise, consider points $p_i\neq\bar{x}_0$ on a minimal geodesic connecting $\bar{x}_0$ and $g_y^{-1}y$ such that $\lim\limits_{i\to\infty}d(g_y^{-1}y,p_i)=0$. Then $p_i\in D_g$ for every $i$ and every $g\neq e$ (since geodesics cannot branch). So $g_y^{-1}y\in \overline{F}$, $y\in g_y\cdot(B_r(\bar{x}_0)\cap \overline{F})$ .
		
	\end{proof}
	Since for every $g_1\neq g_2$ in $G$, $(g_1\cdot(B_r(\bar{x}_0)\cap F))\cap (g_2\cdot(B_r(\bar{x}_0)\cap F))=\emptyset$, taking the volume in Claim \ref{claim1}, we get $\#(D^G(\bar{x}_0,2r))\cdot \text{Vol}(B_r(x_0))\geq \text{Vol}(B_r(\bar{x}_0))$. 
	
	Proof of inequality(\ref{and}): It is clear that $\bigcup\limits_{g\in D(\bar{x}_0,r)}g\cdot (B_r(\bar{x}_0)\cap F)\subset B_{2r}(\bar{x}_0)$. The result follows by taking the volume.

\end{proof}

\section*{Acknowledgements}
The author thanks Professor Shicheng Xu for suggesting  the problem of the  first Betti number rigidity of open manifolds with nonnegative Ricci curvature to him.  The author thanks his advisor Professor Xiaochun Rong for pointing out to him the equivalence of condition(c) and (d)  in Theorem \ref{main} and for his sincere guidance.

	\clearpage
	\bibliography{refs}		
\end{document}